\documentclass[11pt]{article}
\usepackage{amsthm,amsfonts, amsbsy, amssymb,amsmath,graphicx}
\usepackage{graphics}
\usepackage[english]{babel}
\usepackage{graphicx}
\usepackage{lineno}
\usepackage{enumerate}
\usepackage{tikz}
\usepackage{tkz-graph}
\usepackage{tkz-berge}
\usepackage{hyperref}
\usepackage{color}
\usepackage{tikz}

\hypersetup{colorlinks=true}

\hypersetup{colorlinks=true, linkcolor=blue, citecolor=blue,urlcolor=blue}

\input epsf
\usepackage{epic,latexsym,amssymb}
\usepackage{tikz}

\usepackage{tikz,epsf}

\usepackage{amsfonts}
\usepackage{amscd}
\usepackage{amsmath}
\usepackage{graphicx}
\usepackage{color}
\usepackage{caption,subcaption}

\newtheorem{theorem}{Theorem}
\newtheorem{observation}[theorem]{Observation}
\newtheorem{proposition}[theorem]{Proposition}

\newtheorem{corollary}[theorem]{Corollary}
\newtheorem{fact}[theorem]{Fact}

\tikzstyle{vertex}=[circle, draw, inner sep=0pt, minimum size=6pt]

\newcommand{\QEDmark}{\mbox{\textsc{qed}}}
\newcommand{\proofStarter}[1]{\textsc{#1} }

\begin{document}

\title{Outer independent double Roman domination number of graphs}
\author{{\small Doost Ali Mojdeh$^1$, Babak Samadi$^2$, Zehui Shao$^3$ and Ismael G. Yero$^4$}\\{\small Department of Mathematics, University of Mazandaran,}
\\{\small Babolsar, Iran$^{1,2}$}\\{\small damojdeh@umz.ac.ir$^1$}, {\small samadibabak62@gmail.com$^2$}\vspace{2mm}
\\{\small Institute of Computing Science and Technology, Guangzhou University, Guangzhou 510006, China$^3$}\\
{\small zshao@gzhu.edu.cn$^3$}\vspace{2mm}
\\{\small Departamento de Matem\'{a}ticas, Universidad de C\'{a}diz, Algeciras, Spain$^4$}\\
{\small ismael.gonzalez@uca.es$^4$}
}
\date{}
\maketitle

\begin{abstract}

A double Roman dominating function of a graph $G$ is a function $f:V(G)\rightarrow \{0,1,2,3\}$ having the property that for each vertex $v$ with $f(v)=0$, there exists $u\in N(v)$ with $f(u)=3$, or there are $u,w\in N(v)$ with $f(u)=f(w)=2$, and if $f(v)=1$, then $v$ is adjacent to a vertex assigned at least $2$ under $f$. The double Roman domination number $\gamma_{dR}(G)$ is the minimum weight $f(V(G))=\sum_{v\in V(G)}f(v)$ among all double Roman dominating functions of $G$. An outer independent double Roman dominating function is a double Roman dominating function $f$ for which the set of vertices assigned $0$ under $f$ is independent. The outer independent double Roman domination number $\gamma_{oidR}(G)$ is the minimum weight taken over all outer independent double Roman dominating functions of $G$.

In this work, we present some contributions to the study of outer independent double Roman domination in graphs. Characterizations of the families of all connected graphs with small outer independent double Roman domination numbers, and tight lower and upper bounds on this parameter are given. We moreover bound this parameter for a tree $T$ from below by two times the vertex cover number of $T$ plus one. We also prove that the decision problem associated with $\gamma_{oidR}(G)$ is NP-complete even when restricted to planar graphs with maximum degree at most four. Finally, we give an exact formula for this parameter concerning the corona graphs.
\end{abstract}




\section{Introduction and preliminaries}

Throughout this paper, we consider $G$ as a finite simple graph with vertex set $V(G)$ and edge set $E(G)$. We use \cite{w} as a reference for terminology and notation which are not explicitly defined here. The {\em open neighborhood} of a vertex $v$ is denoted by $N(v)$, and its {\em closed neighborhood} is $N[v]=N(v)\cup \{v\}$. The {\em minimum} and {\em maximum degrees} of $G$ are denoted by $\delta(G)$ and $\Delta(G)$, respectively. The {\em corona} of two graphs $G_{1}$ and $G_{2}$ is the graph $G_{1}\odot G_{2}$ formed from one copy of $G_{1}$ and $|V(G_{1})|$ copies of $G_{2}$ where the $i$th vertex of $G_{1}$ is adjacent to every vertex in the $i$th copy of $G_{2}$. For a function $f:V(G)\rightarrow\{0,\cdots,k\}$ we let $V^{f}_{i}=\{v\in V(G)\mid f(v)=i\}$, for each $0\leq i\leq k$ (we simply write $V_{i}$ if there is no ambiguity with respect to the function $f$). We call $\omega(f)=f(V(G))=\sum_{v\in V(G)}f(v)$ as the {\em weight} of $f$.

A set $S\subseteq V(G)$ of $G$ is called a {\em dominating set} if every vertex not in $S$ has a neighbor in $S$. The {\em domination number} $\gamma(G)$ of $G$ is the minimum cardinality among all dominating sets of $G$. A subset $I\subseteq V(G)$ is said to be {\em independent} if no two vertices in $I$ are adjacent. The {\em independence number} $\alpha(G)$ is the maximum cardinality among all independent sets of $G$. A {\em vertex cover} of $G$ is a set $Q\subseteq V(G)$ that contains at least one endpoint of every edge. The {\em vertex cover number} $\beta(G)$ is the minimum cardinality among all vertex cover sets of $G$. For any parameter $p$ of $G$, by a $p(G)$-set we mean a set of cardinality $p(G)$.

A {\em Roman dominating function} of a graph $G$ is a function $f:V(G)\rightarrow\{0,1,2\}$ such that if $v\in V_0$ for some $v\in V(G)$, then there exists $w\in N(v)$ such that $w\in V_2$. The minimum weight of a Roman dominating function $f$ of $G$ is called the {\em Roman domination number} of $G$, denoted by $\gamma_{R}(G)$. This concept was formally defined by Cockayne \emph{et al.} \cite{cdhh} motivated, in some sense, by the article of Ian Stewart entitled ``Defend the Roman Empire!" (\cite{s}), published in {\it Scientific American}. The idea is that the values $1$ and $2$ represent the number of Roman legions stationed at a location $v$. A location $u\in N(v)$ is considered to be {\em unsecured} if no legion is stationed there ($f(u)=0$). The unsecured location $u$ can be secured by sending a legion to $u$ from an adjacent location $v$. But a legion cannot be sent from a location $v$ if doing so leaves that location unsecured (if $f(v)=1$). Thus, two legions must be stationed at a location ($f(v)=2$) before one of the legions can be sent to an adjacent location.

Once the seminal paper \cite{cdhh} was published, this topic attracted the attention of many researchers. One of the research lines that has recently become popular concerns variations in the concept of Roman domination involving some vertex independence features. Some of these variations have been outlined in \cite{cky}.

For instance, an {\em outer independent Roman dominating function} (OIRD function) of a graph $G$ is a Roman dominating function $f:V(G)\rightarrow\{0,1,2\}$ for which $V^{f}_{0}$ is independent. The {\em outer independent Roman domination number} (OIRD number) $\gamma_{oiR}(G)$ is the minimum weight of an OIRD function of $G$. This parameter was introduced in \cite{acs1}. A total domination version of such parameter above was presented in \cite{cky}.

On the other hand, Beeler \emph{et al}. \cite{bhh} introduced the concept of double Roman domination. This provided a stronger and more flexible level of defense in which three legions can be deployed at a given location. They also presented some real privileges of this concept in comparison with the Roman domination. But existing two adjacent locations with no legions can jeopardize them. Indeed, they would be considered more vulnerable. So, one improved situation for a location with no legion is to be surrounded by locations in which legions are stationed. This motivates us to consider a double Roman dominating function $f$ for which $V^{f}_{0}$ is an independent sets, which is the concept that will be investigated in this paper. More formally, a {\em double Roman dominating function} (DRD function for short) of a graph $G$ is a function $ f:V(G)\rightarrow \{0,1,2,3\}$ for which the following conditions are satisfied.
\begin{itemize}
  \item[(a)] If $f(v)=0$, then the vertex $v$ must have at least two neighbors in $V_2$ or one neighbor in $V_3$.
  \item[(b)] If $f(v)=1$ , then the vertex $v$ must have at least one neighbor in $V_2\cup V_3$.
\end{itemize}
This parameter was also studied in \cite{al}, \cite{jr} and \cite{zljs}. Accordingly, an {\em outer independent double Roman dominating function} (OIDRD function for short) is a DRD function for which $V^{f}_{0}$ is independent. The {\em \emph{(}outer independent\emph{)} double Roman domination number} ($\gamma_{oidR}(G)$) $\gamma_{dR}(G)$ equals the minimum weight of (an) a (OIDRD function) DRD function of $G$. This concept was first introduced in \cite{acss}.

For the sake of convenience, an OIDRD function (OIRD function) $f$ of a graph $G$ with weight $\gamma_{oidR}(G)$ ($\gamma_{oiR}(G)$) is called a $\gamma_{oidR}(G)$-function ($\gamma_{oiR}(G)$-function).

In this paper, we characterize the families of all connected graphs $G$ with small OIDRD numbers (that is $\gamma_{oidR}(G)\in\{3,4,5\}$), and give tight lower and upper bounds on this parameter in terms of several well-known graph parameters. We also prove that the decision problem associated with $\gamma_{oidR}(G)$ is NP-complete for planar graphs with maximum degree at most four.

We begin with some easily verified facts about the OIDRD numbers of some basic families of graphs.

\begin{observation}\label{ob1}The following statements hold.
\begin{itemize}
  \item[{\rm (i)}] For $n\geq1$,
$\gamma_{oidR}(P_n)=\left\{
\begin{array}
[l]{ll}%
n,& \text{if }\ n=3,\\
n+1, & \text{if }\ n\neq3.
\end{array}(\emph{\cite{acss}})
\right.$
  \item[{\rm (ii)}] For $n\geq3$,
$\gamma_{oidR}(C_n)=\left\{
\begin{array}
[l]{ll}%
n, & \text{if }\ n\ \text{is}\  \text{even},\\
n+1, & \text{if }\ n\ \text{is}\  \text{odd}.
\end{array}(\emph{\cite{acss}})
\right.$
  \item[{\rm (iii)}] For $n\ge 1$, $\gamma_{oidR}(K_n)=n+1$. \emph{(\cite{acss})}
  \item[{\rm (iv)}] For positive integers $m\le n$,
$\gamma_{oidR}(K_{m,n})=\left\{
\begin{array}
[l]{lll}%
3, & \text{if }\ m=1,\\
2m, & \text{if }\ m\in\{2,3\},\\
m+4, &  \text{otherwise}.
\end{array}
\right.$
  \item[{\rm (v)}] For a complete $k$-partite graph $K_{n_1,n_2,...,n_k}$ with $k\geq3$ and $1\le n_1\le n_2\le \ldots\le n_k$, $\gamma_{oidR}(K_{n_1,n_2,...,n_k})=\sum_{i=1}^{k-1}n_i+2$.
\end{itemize}
\end{observation}

While considering a DRD function $f=(V_0,V_1,V_2,V_3)$ one can assume that $V_1=\emptyset$ (see \cite{bhh}). In contrast, OIDRD function behave a little different. For instance, if $G=K_{m,n}$ with $5\le m \le n$, then $\gamma_{oidR}(K_{m,n})=m+4$ and all vertices of the smaller partite set have positive values. This shows that some vertices of the smaller partite set are assigned inevitably the value $1$. That is stated in the following observation.

\begin{fact}
Let $f=(V_0,V_1,V_2,V_3)$ be an OIDRD function of a graph $G$. Then, $V_1$ is not necessarily empty.
\end{fact}


\section{Connected graphs with small OIDRD numbers}

In this section, we characterize the family of all connected graphs $G$ for which $\gamma_{oidR}(G)\in \{3,4,5\}$. To this end, let $\mathcal{G}$ be the family of all graphs of the form $G_1$, $G_2$ and $G_3$ depicted in Figure \ref{fig:g1-g2-g3}. In the figure, the number of vertices $w_1,\cdots w_k$ in $G_1$, $G_2$, and $G_3$ is at least $1$, $1$ and $2$, respectively.

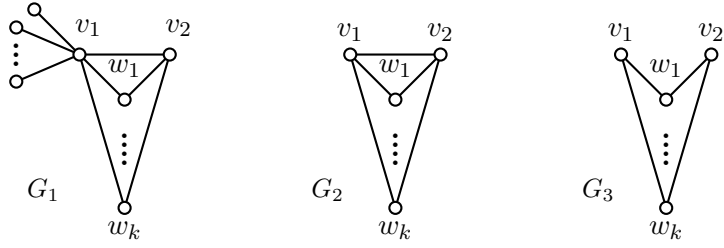
\begin{figure}[h]
\tikzstyle{every node}=[circle, draw, fill=white!, inner sep=0pt,minimum width=.16cm]
\begin{center}
\begin{tikzpicture}[thick,scale=.6]
  \draw(0,0) { 


+(-12,0) node{}

    +(-10,0) node{}
    +(-11,-1) node{}

    +(-11,-3.4) node{}
    +(-12,0) -- +(-10,0)
    +(-13,1) node{}
    +(-13.4,.6) node{}
    +(-13.4,-.6) node{}
    +(-12,0) -- +(-13,1)
    +(-12,0) -- +(-13.4,.6)
    +(-12,0) -- +(-13.4,-0.6)

+(-11,-1) -- +(-12,0)
+(-11,-1) -- +(-10,0)
 +(-11,-3.4) -- +(-12,0)
  +(-11,-3.4) -- +(-10,0)

+(-11,-1.8) node[rectangle, draw=white!0, fill=white!100]{ ${\textbf{.}}$}
+(-11,-2) node[rectangle, draw=white!0, fill=white!100]{ ${\textbf{.}}$}
+(-11,-2.2) node[rectangle, draw=white!0, fill=white!100]{ ${\textbf{.}}$}
+(-11,-2.4) node[rectangle, draw=white!0, fill=white!100]{ ${\textbf{.}}$}
+(-11.8,.6) node[rectangle, draw=white!0, fill=white!100]{ ${v_1}$}
+(-9.8,.6) node[rectangle, draw=white!0, fill=white!100]{ ${v_2}$}

+(-13.4,.2) node[rectangle, draw=white!0, fill=white!100]{ ${\textbf{.}}$}
+(-13.4,0) node[rectangle, draw=white!0, fill=white!100]{ ${\textbf{.}}$}
+(-13.4,-.2) node[rectangle, draw=white!0, fill=white!100]{ ${\textbf{.}}$}


+(-11,-.3) node[rectangle, draw=white!0, fill=white!100]{${w_1}$}
+(-11,-3.9) node[rectangle, draw=white!0, fill=white!100]{ ${w_k}$}

+((-12.8,-3) node[rectangle, draw=white!0, fill=white!100]{{\small $G_1$}}



   +(-6,0) node{}

    +(-4,0) node{}
    +(-5,-1) node{}

    +(-5,-3.4) node{}
    +(-6,0) -- +(-4,0)

+(-5,-1) -- +(-6,0)
+(-5,-1) -- +(-4,0)
 +(-5,-3.4) -- +(-6,0)
  +(-5,-3.4) -- +(-4,0)

+(-5,-1.8) node[rectangle, draw=white!0, fill=white!100]{ ${\textbf{.}}$}
+(-5,-2) node[rectangle, draw=white!0, fill=white!100]{ ${\textbf{.}}$}
+(-5,-2.2) node[rectangle, draw=white!0, fill=white!100]{ ${\textbf{.}}$}
+(-5,-2.4) node[rectangle, draw=white!0, fill=white!100]{ ${\textbf{.}}$}
+(-6,.5) node[rectangle, draw=white!0, fill=white!100]{ ${v_1}$}
+(-4,.5) node[rectangle, draw=white!0, fill=white!100]{ ${v_2}$}

+(-5,-.3) node[rectangle, draw=white!0, fill=white!100]{${w_1}$}
+(-5,-3.9) node[rectangle, draw=white!0, fill=white!100]{ ${w_k}$}

+((-6.5,-3) node[rectangle, draw=white!0, fill=white!100]{{\small $G_2$}}



 +(0,0) node{}

    +(2,0) node{}
    +(1,-1) node{}

    +(1,-3.4) node{}

+(1,-1) -- +(0,0)
+(1,-1) -- +(2,0)
 +(1,-3.4) -- +(0,0)
  +(1,-3.4) -- +(2,0)

+(1,-1.8) node[rectangle, draw=white!0, fill=white!100]{ ${\textbf{.}}$}
+(1,-2) node[rectangle, draw=white!0, fill=white!100]{ ${\textbf{.}}$}
+(1,-2.2) node[rectangle, draw=white!0, fill=white!100]{ ${\textbf{.}}$}
+(1,-2.4) node[rectangle, draw=white!0, fill=white!100]{ ${\textbf{.}}$}
+(0,.5) node[rectangle, draw=white!0, fill=white!100]{ ${v_1}$}
+(2,.5) node[rectangle, draw=white!0, fill=white!100]{ ${v_2}$}

+(1,-.3) node[rectangle, draw=white!0, fill=white!100]{ ${w_1}$}
+(1,-3.9) node[rectangle, draw=white!0, fill=white!100]{ ${w_k}$}

+((-0.5,-3) node[rectangle, draw=white!0, fill=white!100]{{\small $G_3$}}

 };
\end{tikzpicture}
\end{center}
\caption{The graphs $G_1$, $G_2$ and $G_3$.}\label{fig:g1-g2-g3}
\end{figure}
We next define other six necessary families of graphs, that is, the families $\mathcal{H}_{i}$, $1\leq i\leq 6$. To this end, we shall use the following conventions. For a given set of vertices $\{v_1,\dots,v_r\}$ with $r\ge 1$, by $V_{v_1,\dots,v_r}$ we represent another disjoint set of vertices such that every vertex $v\in V_{v_1,\dots,v_r}$ satisfies $N(v)=\{v_1,\dots,v_r\}$. Such convention shall be used also while proving Proposition \ref{prop1}.
\begin{itemize}
  \item $\mathcal{H}_{1}$: We begin with a path $P=abc$. Then we add four sets $V_{b}$, $V_{a,b}$, $V_{b,c}$ and $V_{a,b,c}$ such that one of the following conditions holds.

      - ($a_{1}$): $V_{a,b},V_{b,c}=\emptyset$ and $|V_{a,b,c}|\geq 2$.

      - ($b_{1}$) only one of the sets $V_{a,b}$ and $V_{b,c}$ is empty, and $V_{a,b,c}\neq \emptyset$.

      - ($c_{1}$) $V_{a,b},V_{b,c}\neq \emptyset$.
  \item $\mathcal{H}_{2}$: We begin with a cycle of order three $C=abca$ and proceed as above, by adding the sets $V_{b}$, $V_{a,b}$, $V_{b,c}$ and $V_{a,b,c}$. Then, one of the following situations holds.

   - ($a_{2}$) $V_{a,b,c}\neq \emptyset$.

   - ($b_{2}$) $V_{a,b},V_{b,c}\neq \emptyset$.
  \item $\mathcal{H}_{3}$: We begin with two nonadjacent vertices $a$ and $b$, and add the non-empty sets of vertices $V_{a}$ and $V_{a,b}$.
  \item $\mathcal{H}_{4}$: We begin with a vertex $a$ and an edge $bc$. Then we add the sets $V_{a,b}$ and $V_{a,b,c}$ such that one of the following conditions holds.

      - ($a_{4}$) $V_{a,b}=\emptyset$ and $|V_{a,b,c}|\geq2$.

      - ($b_{4}$) $V_{a,b}\neq \emptyset$.
  \item $\mathcal{H}_{5}$: We begin with a path $P=abc$. Then we add the sets $V_{a,b}$ and $V_{a,b,c}$ such that one of the following conditions holds.

      - ($a_{5}$) $V_{a,b},V_{a,b,c}\neq \emptyset$.

      - ($b_{5}$) $V_{a,b}=\emptyset$ and $|V_{a,b,c}|\geq2$.
  \item $\mathcal{H}_{6}$: We begin with a path $P=abc$. Then we add the sets $V_{a,c}$ (note that for any vertex $v\in V_{a,c}$, it happens $N(v)=\{a,c\}$) and $V_{a,b,c}$, such that one of the next conditions holds.

  - ($a_{6}$) $V_{a,c},V_{a,b,c}\neq \emptyset$.

  - ($b_{6}$) $V_{a,c}=\emptyset$ and $|V_{a,b,c}|\geq2$.
\end{itemize}

\begin{proposition}\label{prop1}
Let $G$ be a connected graph of order $n\ge 3$. Then,
\begin{itemize}
  \item[{\rm (i)}] $\gamma_{oidR}(G)=3$  if and only if $G$ is a star.
  \item[{\rm (ii)}] $\gamma_{oidR}(G)=4$ if and only if $G \in \mathcal{G}$.
  \item[{\rm (iii)}] $\gamma_{oidR}(G)=5$ if and only if $G\in \cup_{i=1}^{6}\mathcal{H}_{i}$.
\end{itemize}
\end{proposition}

\begin{proof}
(i) It is clear.

(ii) Let $G\in \mathcal{G}$. In $G_1$, if we assign the value $3$ to the vertex $v_1$, the value $1$ to the vertex $v_2$ and $0$ to the other vertices, then we have $\gamma_{oidR}(G_1)\le 4$. In $G_2$, if we assign $2$ to the vertices $v_1$ and $v_2$ (or $1$ to one of them and $3$ to the other one), then $\gamma_{oidR}(G_2)\le 4$. In $G_3$, if we assign $2$ to the vertices $v_1$ and $v_2$, then $\gamma_{oidR}(G_3)\le 4$. Since, $G_1$, $G_2$, and $G_3$ are not stars, by item (i), we have the equality.

Conversely, let $G$ be a graph with $\gamma_{oidR}(G)=4$ and let $f=(V_0,V_1,V_2,V_3)$ be a $\gamma_{oidR}(G)$-function. We first note that the case $(|V_{0}|,|V_{1}|,|V_{2}|,|V_{3}|)=(0,2,1,0)$ is possible if and only if $G\cong K_{3}$. So, $G$ is of the form of $G_{2}$. Consequently, we have only two remaining possibilities.

- ($a$) There are two adjacent vertices $v_1$ and $v_2$ in $G$ such that $v_1\in V_3$ and $v_2\in V_1$ or $\{v_1,v_2\}\subseteq V_2$ and the other $k\geq1$ vertices are independent and belong to $V_0$ (note that $k$ must be at least one, for otherwise $G$ would be a star with $\gamma_{oidR}(G)=3$). Therefore, such a graph should be of the form $G_1$ or $G_2$ in $\mathcal{G}$.

- ($b$) There are two non adjacent vertices $v_1$ and $v_2$ in $G$ such that $\{v_1,v_2\}\subseteq V_2$ and the remaining vertices are independent and belong to $V_0$. Therefore, such a graph should be a graph like $G_3$ in $\mathcal{G}$. Note that in such a case we have $k\geq2$, for otherwise $G$ is disconnected or satisfies that $\gamma_{oidR}(G)\le 3$.

(iii) If $G\in \mathcal{H}_{1}\cup \mathcal{H}_{2}$, then $(f(a),f(b),f(c))=(1,3,1)$ and $f(v)=0$ otherwise, is an OIDRD function of $G$ that leads to $\gamma_{oidR}(G)\le 5$. If $G\in \mathcal{H}_{3}$, then $(f(a),f(b))=(3,2)$ and $f(v)=0$ for any other vertex, is an OIDRD function of $G$ that gives $\gamma_{oidR}(G)\le 5$. Also, if $G\in \mathcal{H}_{4}\cup \mathcal{H}_{5}$, then $(f(a),f(b),f(c))=(2,2,1)$ and $f(v)=0$ otherwise, is an OIDRD function of $G$, and so $\gamma_{oidR}(G)\leq5$. Finally, if $G\in \mathcal{H}_{6}$, then $(f(a),f(b),f(c))=(2,1,2)$ and $f(v)=0$ for any other vertex, is a desired OIDRD function of $G$ that gives the same conclusion as above. Since the graphs of the family $\cup_{i=1}^{6}\mathcal{H}_{i}$ neither are stars nor are included in the family $\mathcal{G}$, by items (i) and (ii), we get the desired equalities.

Conversely, we assume that $f:V(G)\rightarrow\{0,1,2,3\}$ is a $\gamma_{oidR}(G)$-function of weight $5$. If $V_{1}=\emptyset$, then there exist two vertices $a$ and $b$ such that $(f(a),f(b))=(3,2)$. Note that $f$ assigns $0$ to the other vertices and that $ab\notin E(G)$, necessarily. In such situation, since $G$ is connected, at least one vertex must be adjacent to $b$ and each such vertex must be adjacent to $a$, as well. Now if $V_{a}=\emptyset$, then $\gamma_{oidR}(G)\le 4$, which is a contradiction. This shows that $G\in \mathcal{H}_{3}$.

We now assume that $V_{1}\neq \emptyset$. Suppose that $|V_{1}|=1$ and $b$ is the only member of $V_{1}$. Therefore, there are two vertices $a$ and $c$ assigned $2$ under $f$. We first consider $b$ is adjacent to both $a$ and $c$. Note that the remaining vertices must be adjacent to both $a$ and $c$, as well. If $V_{a,c}=\emptyset$ and $|V_{a,b,c}|\leq 1$, then we have $\gamma_{oidR}(G)\leq 4$. Thus, $|V_{a,b,c}|\geq 2$. If $V_{a,c}\neq \emptyset$ and $V_{a,b,c}=\emptyset$, then we have $\gamma_{oidR}(G)\le 4$. Hence, $V_{a,b,c}\neq \emptyset$. This shows that $G\in \mathcal{H}_{6}$.

Let $b$ be adjacent to only one vertex in $\{a,c\}$, say $c$. We deal with two possibilities depending on the adjacency between $a$ and $c$. First, let $ac\in E(G)$. Then, the other vertices belong to $V_{a,c}\cup V_{a,b,c}$. If $V_{a,c}=\emptyset$ and $|V_{a,b,c}|\leq1$, then $\gamma_{oidR}(G)\leq4$, and so $|V_{a,b,c}|\geq2$. If $V_{a,c}\neq \emptyset$ and $V_{a,b,c}=\emptyset$, then $\gamma_{oidR}(G)\le 4$. Therefore, $V_{a,b,c}\neq \emptyset$. In such case, $G\in \mathcal{H}_{5}$. Now let $ac\notin E(G)$. Thus, the other vertices belong to $V_{a,c}\cup V_{a,b,c}$. If $V_{a,c}=\emptyset$ and $|V_{a,b,c}|\leq1$, then $G$ is disconnected or $\gamma_{oidR}(G)\le 4$. Therefore, $|V_{a,b,c}|\geq2$. Note that if $V_{a,c}\neq\emptyset$, we have no conditions on the set $V_{a,b,c}$. Consequently, $G\in \mathcal{H}_{4}$.

We now consider a situation in which $|V_{1}|=2$. Let $V_{1}=\{a,c\}$. Then, both $a$ and $c$ must be adjacent to a vertex $b$ assigned $3$ under $f$. Hence, the other vertices belong to $V_{b}\cup V_{a,b}\cup V_{b,c}\cup V_{a,b,c}$. We need to consider two possibilities depending on the adjacency between $a$ and $c$. First, let $ac\notin E(G)$ and assume that $V_{a,b}=V_{b,c}=\emptyset$. If $|V_{a,b,c}|\leq1$, then we have $\gamma_{oidR}(G)\leq4$, and so $|V_{a,b,c}|\geq2$. If only one of the sets $V_{a,b}$ and $V_{b,c}$ is empty, and $V_{a,b,c}=\emptyset$, then $\gamma_{oidR}(G)\le 4$. Thus, $V_{a,b,c}\neq \emptyset$. We now note that if $V_{a,b},V_{b,c}\neq \emptyset$, then we have no conditions on the set $V_{a,b,c}$. This argument guarantees that $G\in \mathcal{H}_{1}$. On the other hand, let $ac\in E(G)$. Hence, we have a cycle $abca$. If at least one of the sets $V_{a,b}$ and $V_{b,c}$ is empty, then we must have $V_{a,b,c}\neq \emptyset$, for otherwise $\gamma_{oidR}(G)\le 4$. If both $V_{a,b}$ and $V_{b,c}$ are nonempty, then we have no conditions on the set $V_{a,b,c}$. Therefore, $G\in \mathcal{H}_{2}$.

Finally, in the case $|V_{1}|=3$ we have $V_{0}=V_{3}=\emptyset$ and only one vertex is assigned $2$ under $f$. In such situation, $G\cong K_{4}\in \mathcal{H}_{2}$. This completes the proof.
\end{proof}


\section{Computational and combinatorial results}

We first consider the problem of deciding whether a graph $G$ has the OIDRD number at most a given integer. That is stated in the following decision problem. Note that Ahangar \emph{et al}. \cite{acss} proved that the problem of computing the OIDRD number of graphs is NP-hard, even when restricted to bipartite graphs and chordal graphs.

$$\begin{tabular}{|l|}
  \hline
  \mbox{OIDRD problem}\\
  \mbox{INSTANCE: A graph $G$ and an integer $k\leq2|V(G)|$.}\\
  \mbox{QUESTION: Is $\gamma_{oidR}(G)\leq k$?}\\
  \hline
\end{tabular}$$

Our aim is to show that the problem is NP-complete for planar graphs with maximum degree at most four. To this end, we make use of the well-known INDEPENDENCE NUMBER PROBLEM (IN problem) which is known to be NP-complete from \cite{gj}.

$$\begin{tabular}{|l|}
  \hline
  \mbox{IN problem}\\
  \mbox{INSTANCE: A graph $G$ and an integer $k\leq|V(G)|$.}\\
  \mbox{QUESTION: Is $\alpha(G)\geq k$?}\\
  \hline
\end{tabular}$$

Moreover, the problem above remains NP-complete even when restricted to some planar graphs. Indeed, we have the following result.

\begin{theorem}\emph{(\cite{gj})}
The IN problem is NP-complete even when restricted to planar graphs of maximum degree at most three.
\end{theorem}

\begin{theorem}\label{planar}
The OIDRD problem is NP-complete even when restricted to planar graphs with maximum degree at most four.
\end{theorem}

\begin{proof}
Let $G$ be a planar graph with $V(G)=\{v_{1},\dots,v_{n}\}$ and maximum degree $\Delta(G)\leq3$. For any $1\leq i\leq n$, we add a copy of the path $P_{3}$ with central vertex $u_{i}$. We now construct a graph $G'$ by joining $v_{i}$ to $u_{i}$, for each $1\leq i\leq n$. Clearly, $G'$ is a planar graph, $|V(G')|=4n$ and $\Delta(G')\leq4$.

Let $f$ be $\gamma_{oidR}(G')$-function. Since $u_{i}$ is adjacent to two leaves, $f$ must assign a weight of at least three to $u_{i}$ together with the two leaves adjacent to it. So, without loss of generality, we may consider that $f(u_{i})=3$, and that $f$ assigns $0$ to both leaves adjacent to $u_{i}$, for each $1\leq i\leq n$. Since $V_{0}^{f}$ is independent, the number of vertices $v_{i}\in V(G)$ which can be assigned $0$ under $f$ is at most $\alpha(G)$. Furthermore, the other vertices of $V(G)$ are assigned at least $1$ under $f$. Consequently, we obtain that $\gamma_{oidR}(G')\ge 3n+(n-\alpha(G))=4n-\alpha(G)$.

On the other hand, let $I$ be an $\alpha(G)$-set. It is easy to observe that the function
\begin{equation*}
g(v)=\left\{
\begin{array}
[l]{lll}%
3, & \text{if }\ v\in\{u_{1},\cdots,u_{n}\},\\
0, & \text{if }\ $\textit{v}$\ \mbox{is a leaf or \textit{v}$\in I$},\\
1, &  \text{otherwise}.
\end{array}
\right.
\end{equation*}
is an OIDRD function of $G'$ with weight $4n-\alpha(G)$, which leads to the equality $\gamma_{oidR}(G')=4n-\alpha(G)$. Now, by taking $j=4n-k$, it follows that $\gamma_{oidR}(G')\leq j$ if and only if $\alpha(G)\geq k$, which completes the reduction. Since the IN problem is NP-complete for planar graphs of maximum degree at most three, we deduce that the OITRD problem is NP-complete for planar graphs of maximum degree at most four.
\end{proof}

As a consequence of Theorem \ref{planar}, we conclude that the problem of computing the OIDRD number even when restricted to planar graphs with maximum degree at most four in NP-hard. In consequence, it would be desirable to bound the OIDRD number in terms of several different invariants of the graph.

\begin{theorem}
For any graph $G$, $\gamma_{oidR}(G)\le2\gamma_{oiR}(G)$ with equality if and only if $G=\overline{K_n}$.
\end{theorem}
\begin{proof}
If $f=(V_0,V_1,V_2)$ is a $\gamma_{oiR}(G)$-function, it is easy to observe that $g=(V^{g}_0=V_0,V^{g}_1=\emptyset,V^{g}_2=V_1,V^{g}_3=V_2)$ is an OIDRD function of $G$. Therefore,
\begin{equation}\label{EQ1}
\gamma_{oidR}(G)\leq2|V_1|+3|V_2|\leq2|V_1|+4|V_2|=2\gamma_{oiR}(G).
\end{equation}

Clearly, $\gamma_{oidR}(\overline{K_n})=2\gamma_{oiR}(\overline{K_n})=2n$. We now let $\gamma_{oidR}(G)=2\gamma_{oiR}(G)$. This equality along with the inequality chain (\ref{EQ1}) imply that $V_2=\emptyset$, and since $f$ is an OIRD function of $G$, $V_0=V^{g}_0=\emptyset$ as well. Therefore, all vertices of $G$ are assigned $2$ under $g$. Now if there exists an edge $uv$ in $G$, then the function $g'$ assigning $3$ to $u$, $0$ to $v$, and $2$ to the other vertices is an OIDRD function of $G$ with weight less than $\omega(g)$, which is a contradiction. Therefore, $G=\overline{K_n}$.\end{proof}

As an immediate consequence of the equation (\ref{EQ1}), we have the following result.

\begin{corollary}\label{cor2} If $G$ is a connected graph and $f=(V_0,V_1,V_2)$ is a $\gamma_{oiR}(G)$-function, then $\gamma_{oidR}(G)\le2\gamma_{oiR}(G)-|V_2|$.
\end{corollary}

\begin{figure}[h]
\tikzstyle{every node}=[circle, draw, fill=white!, inner sep=0pt,minimum width=.16cm]
\begin{center}
\begin{tikzpicture}[thick,scale=.6]
  \draw(0,0) { 

+(-6,0) node{}

    +(-4,0) node{}

    +(-6,0) -- +(-4,0)
    +(-7,1) node{}
    +(-7.4,.6) node{}
    +(-7.4,-.6) node{}
    +(-6,0) -- +(-7,1)
    +(-6,0) -- +(-7.4,.6)
    +(-6,0) -- +(-7.4,-0.6)

+(-7.4,.2) node[rectangle, draw=white!0, fill=white!100]{ ${\textbf{.}}$}
+(-7.4,0) node[rectangle, draw=white!0, fill=white!100]{ ${\textbf{.}}$}
+(-7.4,-.2) node[rectangle, draw=white!0, fill=white!100]{ ${\textbf{.}}$}

+(-2,0) node{}
+(-4,-2) node{}
+(-4,2) node{}
+(-2,0) node{}

+(-4,2) -- +(-4,0)
+(-2,0) -- +(-4,0)
+(-4,-2) -- +(-4,0)

+(-1,1) node{}
    +(-0.6,.6) node{}
    +(-0.6,-.6) node{}
     +(-2,0) -- +(-1,1)
    +(-2,0) -- +(-0.6,.6)
    +(-2,0) -- +(-0.6,-0.6)

    +(-2,-2) node[rectangle, draw=white!0, fill=white!100]{ ${\textbf{.}}$}
+(-2.2,-2.2) node[rectangle, draw=white!0, fill=white!100]{ ${\textbf{.}}$}
+(-1.8,-1.8) node[rectangle, draw=white!0, fill=white!100]{ ${\textbf{.}}$}

 +(-2,2) node[rectangle, draw=white!0, fill=white!100]{ ${\textbf{.}}$}
+(-2.2,2.2) node[rectangle, draw=white!0, fill=white!100]{ ${\textbf{.}}$}
+(-1.8,1.8) node[rectangle, draw=white!0, fill=white!100]{ ${\textbf{.}}$}

+(-6,2) node[rectangle, draw=white!0, fill=white!100]{ ${\textbf{.}}$}
+(-5.8,2.2) node[rectangle, draw=white!0, fill=white!100]{ ${\textbf{.}}$}
+(-6.2,1.8) node[rectangle, draw=white!0, fill=white!100]{ ${\textbf{.}}$}

+(-6,-2) node[rectangle, draw=white!0, fill=white!100]{ ${\textbf{.}}$}
+(-5.8,-2.2) node[rectangle, draw=white!0, fill=white!100]{ ${\textbf{.}}$}
+(-6.2,-1.8) node[rectangle, draw=white!0, fill=white!100]{ ${\textbf{.}}$}

+(-3,-3) node{}
    +(-3.4,-3.4) node{}
    +(-4.6,-3.4) node{}
     +(-4,-2) -- +(-3,-3)
    +(-4,-2) -- +(-3.4,-3.4)
    +(-4,-2) -- +(-4.6,-3.4)

    +(-3.8,-3.4) node[rectangle, draw=white!0, fill=white!100]{ ${\textbf{.}}$}
+(-4.0,-3.4) node[rectangle, draw=white!0, fill=white!100]{ ${\textbf{.}}$}
+(-4.2,-3.4) node[rectangle, draw=white!0, fill=white!100]{ ${\textbf{.}}$}

+(-3,3) node{}
    +(-3.4,3.4) node{}
    +(-4.6,3.4) node{}
     +(-4,2) -- +(-3,3)
    +(-4,2) -- +(-3.4,3.4)
    +(-4,2) -- +(-4.6,3.4)

    +(-3.8,3.4) node[rectangle, draw=white!0, fill=white!100]{${\textbf{.}}$}
+(-4.0,3.4) node[rectangle, draw=white!0, fill=white!100]{ ${\textbf{.}}$}
+(-4.2,3.4) node[rectangle, draw=white!0, fill=white!100]{ ${\textbf{.}}$}

 +(-3.8,3.4) node[rectangle, draw=white!0, fill=white!100]{${\textbf{.}}$}
+(-4.0,3.4) node[rectangle, draw=white!0, fill=white!100]{ ${\textbf{.}}$}
+(-4.2,3.4) node[rectangle, draw=white!0, fill=white!100]{ ${\textbf{.}}$}

+(-0.6,.2) node[rectangle, draw=white!0, fill=white!100]{ ${\textbf{.}}$}
+(-0.6,0) node[rectangle, draw=white!0, fill=white!100]{ ${\textbf{.}}$}
+(-0.6,-.2) node[rectangle, draw=white!0, fill=white!100]{ ${\textbf{.}}$}


 };
\end{tikzpicture}
\end{center}
\caption{The family of graphs $\mathcal{G'}$.} \label{fig:graph-g'}
\end{figure}
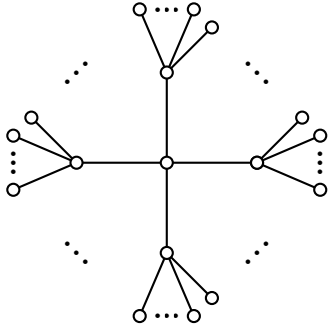

For the equality in the upper bound given in Corollary \ref{cor2}, consider the family of stars, bistars and the family of graphs $\mathcal{G'}$ depicted in Figure \ref{fig:graph-g'}.

\begin{proposition}\label{prop2}
For every graph $G$, $\gamma_{oiR}(G)< \gamma_{oidR}(G)$.
\end{proposition}

\begin{proof} Let $f=(V_0, V_1,V_2,V_3)$ be any $\gamma_{oidR}(G)$-function. If $V_3\ne \emptyset$, then $g=(V^{g}_0=V_0, V^{g}_1=V_1,V^{g}_2=V_2\cup V_3)$ is an OIRD function of $G$, that is, $\gamma_{oiR}(G)< \gamma_{oidR}(G)$. Hence, assume that $V_3 =\emptyset$. Since $V_2 \cup  V_3$ dominates $G$, it follows that $V_2\ne \emptyset$. Thus, all vertices are assigned either the values $0$, $1$ or $2$, and all vertices in $V_0$ must have at least two neighbors in $V_2$ and all vertices in $V_1$ must have at least one neighbor in $V_2$. In such a case, at least one vertex in $V_2$ can be reassigned the value $1$ and the resulting function will be an OIRD function of $G$, as well. Therefore, $\gamma_{oiR}(G)< \gamma_{oidR}(G)$.
\end{proof}

\begin{corollary}\label{cor3}
For any nontrivial connected graph $G$, $\gamma_{oiR}(G)<\gamma_{oidR}(G)<2\gamma_{oiR}(G)$.
\end{corollary}

\begin{theorem}\label{Realize}
For any connected graph $G$ of order $n\geq2$ with maximum degree $\Delta$,
$$max\{\gamma(G),\frac{2}{\Delta}\alpha(G)\}+\beta(G)\leq \gamma_{oidR}(G)\leq3\beta(G).$$
These bounds are sharp.
\end{theorem}

\begin{proof}
Let $I$ be an $\alpha(G)$-set. Hence, the function $f:V(G)\rightarrow\{0,1,2,3\}$ for which $f(v)=0$ if $v\in I$, and $f(v)=3$ for any other vertex, defines an OIDRD function of $G$. Therefore, $\gamma_{oidR}(G)\leq \omega(f)=3(n-\alpha(G))$. Since $\alpha(G)+\beta(G)=n$ (the well known Gallai theorem \cite{G}), the upper bound follows.

That the upper bound is sharp can be seen by the corona $G'\odot \overline{K_{r}}$ for $r\geq2$, in which $G'$ is an arbitrary (connected) graph. Here, $f(v')=3$ for each $v'\in V(G')$, and $f(v)=0$ for all vertices $v$ of the copies of $\overline{K_{r}}$ leads to an OIDRD function of minimum weight in $G$ equals to $3\beta(G)$.

On the other hand, let $g$ be a $\gamma_{oidR}(G)$-function. The set $V_{0}$ is independent and $V_{2}\cup V_{3}$ is a dominating set in $G$, by the properties of an OIDRD function of $G$. Moreover, we have $\omega(g)=|V_{1}|+2|V_{2}|+3|V_{3}|$. These lead to
$$\alpha(G)\geq|V_{0}|=n-(|V_{1}|+|V_{2}|+|V_{3}|)=n-\omega(g)+|V_{2}|+2|V_{3}|\geq n-\omega(g)+|V_{2}|+|V_{3}|\geq n-\omega(g)+\gamma(G).$$
Therefore,
\begin{equation}\label{EQ10}
\gamma_{oidR}(G)=\omega(g)\geq \gamma(G)+\beta(G).
\end{equation}

The lower bound is obvious for $\Delta=1$. So, we assume that $\Delta\geq2$. Now, let $f=(V_{0},V_{1},V_{2},V_{3})$ be a $\gamma_{oidR}(G)$-function. Let $S=V_{0}\cap N(V_{3})$ and $S'=V_{0}\setminus S$. Since each vertex in $V_{3}$ has at most $\Delta$ neighbors in $S$, we have $|S|\leq \Delta|V_{3}|$. Moreover, every vertex in $S'$ has at least two neighbors in $V_{0}$ and every vertex in $V_{0}$ has at most $\Delta$ neighbors in $S'$. Therefore, $2|S'|\leq \Delta|V_{2}|$. The last two inequalities show that $2|V_{0}|=2|S|+2|S'|\leq(|V_{2}|+2|V_{3}|)\Delta$. Taking into account this inequality and since $V_{0}$ is independent, we have
\begin{equation*}
\begin{array}{lcl}
\Delta \gamma_{oidR}(G)=\Delta(|V_{1}|+2|V_{2}|+3|V_{3}|)&=&\Delta(|V_{1}|+|V_{2}|+|V_{3}|)+\Delta(|V_{2}|+2|V_{3}|)\\
&\geq& \Delta(n-|V_{0}|)+2|V_{0}|\geq \Delta n-(\Delta-2)\alpha(G).
\end{array}
\end{equation*}
This implies the lower bound $\gamma_{oidR}(G)\geq n-(\Delta-2)\alpha(G)/\Delta$. Using the equality $\alpha(G)+\beta(G)=n$ again, we have
\begin{equation}\label{EQ11}
\gamma_{oidR}(G)\geq \frac{2}{\Delta}\alpha(G)+\beta(G).
\end{equation}
The desired lower bound now follows from (\ref{EQ10}) and (\ref{EQ11}).

That the lower bound (\ref{EQ10}) is sharp can be seen as follows. Given a positive integer $t$ and $1\leq i\leq t$, let $H_{i}$ be a graph obtained from the complete bipartite graph $K_{2,m_{i}}$ ($m_{i}\geq2$) by adding a new vertex $z_{i}$ and joining it the two vertices, say $x_{i}$ and $y_{i}$, of the smallest partite set of $K_{2,m_{i}}$. We now form a cycle on the set of vertices $\{z_{1},\cdots,z_{t}\}$, and denote the obtained graph by $H$. It is easily observed that $h:V(H)\rightarrow\{0,1,2,3\}$  defined by $f(x_{i})=f(y_{i})=2$ and $f(z_{2i-1})=1$ for all $1\leq i\leq \lceil t/2\rceil$, and $f(v)=0$ for any other vertex, is an OIDRD function of $H$ with minimum weight $4t+\lceil t/2\rceil$. On the other hand, $\beta(H)=3t-\lfloor t/2\rfloor$ and $\gamma(H)=2t$. Therefore, the lower bound (\ref{EQ10}) holds with equality for $H$. Moreover, the lower bound (\ref{EQ11}) is sharp for the star $K_{1,n-1}$. This completes the proof.
\end{proof}

Note that the upper bound given in the theorem above was also given in \cite{acss}. For the sake of completeness, we pointed it out and gave an infinite family of graphs for which the equality holds.


\section{Trees}

The authors of \cite{acss} proved that $\beta(G)+2$ is a lower bound on the OIDRD number of a nontrivial connected graph $G$. This lower bound can be improved for trees. Recall that a {\em double star $S_{a,b}$} is a tree with exactly two non-leaf vertices in which one support vertex is adjacent to $a$ leaves and the other to $b$ leaves.

\begin{theorem}\label{induction}
For any tree $T$, $\gamma_{oidR}(T)\geq2\beta(T)+1$ and this bound is tight.
\end{theorem}

\begin{proof}
We proceed by induction on the order $n$ of $T$. The result is obvious when $n=1$. Moreover, it is easily observed that $\gamma_{oidR}(K_{1,n})=2\beta(K_{1,n})+1=3$. Hence, we may assume that $T$ has diameter $diam(T)\geq3$. If $diam(T)=3$, then $T$ is isomorphic to the double star $S_{a,b}$, $1\leq a\leq b$. We then have $\gamma_{oidR}(S_{1,b})=2\beta(S_{1,b})+1=5$, and $\gamma_{oidR}(S_{a,b})=6>5=2\beta(S_{a,b})+1$ when $a\geq2$. Thus, in what follows we consider that $diam(T)\geq4$, which implies that $n\geq5$.

Suppose that $\gamma_{oidR}(T')\geq2\beta(T')+1$, for each tree $T'$ of order $1\leq n'<n$. Let $T$ be a tree of order $n$. We consider two cases depending on the behavior of support vertices of $T$.

\textit{Case 1.} $T$ has a strong support vertex $u$. Let $v$ be a leaf adjacent to $u$. Consider the tree $T'=T-v$. Note that every $\gamma_{oidR}(T)$-function $f$ assigns $3$ to $u$ and $0$ to the leaves adjacent to $u$, necessarily. It is easy to see that $\beta(T')=\beta(T)$ and that $\gamma_{oidR}(T')\leq \gamma_{oidR}(T)$. Therefore, $\gamma_{oidR}(T)\geq 2\beta(T)+1$ by the induction hypothesis.

\textit{Case 2.} All support vertices of $T$ are weak. Let $r$ and $v$ be two leaves with $d(r,v)=diam(T)$. We
root the tree $T$ at $r$. Let $w$ be the parent of $v$, and $x$ be the parent of $w$. Since $T$ has no strong support vertices, it follows that $w$ has degree $deg(w)=2$. We need to consider two subcases depending on $deg(x)$.

\textit{Subcase 2.1.} $deg(x)\geq3$. Since $d(r,v)=diam(T)$, all children of $x$ are leaves or support vertices. Let $T'=T-T_{w}$ (for a vertex $u$, by $T_{u}$ we mean the subtree of $T$ rooted at $u$ consisting of $u$ and all its descendants in $T$). It is easily observed that $\beta(T)=\beta(T')+1$. Let $f$ be a $\gamma_{oidR}(T)$-function of $T$. If $f(x)\geq2$, then $f(w)+f(v)=2$. Therefore, the restriction of $f$ to $V(T')$, from now on denoted $f'=f\mid_{V(T')}$, is an OIDRD function of $T'$. So, $\gamma_{oidR}(T')\leq w(f')=\gamma_{oidR}(T)-2$. Therefore, $2\beta(T)+1=2\beta(T')+3\leq \gamma_{oidR}(T')+2\leq\gamma_{oidR}(T)$.

Suppose that $f(x)=0$. We may assume, without loss of generality, that $f(v)=0$ and $f(w)=3$. If $x$ is the parent of a support vertex $w'$ different from $w$, then we may assume that $f$ assigns $3$ to $w'$ and $0$ to the leaf adjacent to $w'$. In such a case, $f'=f\mid_{V(T')}$ is an OIDRD function of $T'$ with weight $\omega(f')=\gamma_{oidR}(T)-3$. So, $2\beta(T)+1<\gamma_{oidR}(T)$ by a similar fashion. We now assume that all children of $x$ different from $w$ are leaves. Since $T$ has no strong support vertices, it follows that $x$ is adjacent to only one leaf $x'$. If $f(x')=3$, then $f'$ is an OIDRD function of $T'$ and we are done. So, we may assume that $f(x')=2$. In such a situation, the assignment $(g(x'),g(x),g(w),g(v))=(0,3,0,2)$ and $g(u)=f(u)$ for the other vertices is a $\gamma_{oidR}(T)$-function of $T$. Moreover, $g'=g\mid_{V(T')}$ is an OIDRD function of $T'$ with weight $\omega(g')=\gamma_{oidR}(T)-2$. Hence, we have $2\beta(T)+1\leq \gamma_{oidR}(T)$, by a similar fashion.

Let $f(x)=1$. Since $f(w)+f(v)=3$, we assume that $f(w)=3$ and $f(v)=0$. Suppose that $x$ is adjacent to a leaf $x'$ which is unique since $T$ has no strong support vertices. Then $f(x')\ge 2$, necessarily. Now the assignment $(g(x'),g(x),g(w),g(v))=(0,3,0,2)$ and $g(u)=f(u)$ for the remaining vertices, is an OIDRD function of $T$ with weight less than $\omega(f)$, which is impossible. Therefore, all children of $x$ are support vertices. Let $w'\neq w$ be a child of $x$ adjacent to the leaf $w''$. Since $f(w')+f(w'')=3$, we assume that $f(w')=3$ and $f(w'')=0$. In such a situation, the assignment $(g(w''),g(w'),g(x),g(w),g(v))=(2,0,2,0,2)$ and $g(u)=f(u)$ otherwise, defines an OIDRD function of $T$ with weight less than $\omega(f)$, a contradiction. 

\textit{Subcase 2.2.} $deg(x)=2$. Again, we let $T'=T-T_{w}$. Suppose that $y$ is the parent of $x$. If $f(x)\in\{2,3\}$, then $f(w)=0$ and $f(v)=2$. Therefore, $f'=f\mid_{V(T')}$ is an OIDRD function of $T'$. This shows that $2\beta(T)+1=2\beta(T')+3\leq \gamma_{oidR}(T')+2\leq \omega(f')+2=\gamma_{oidR}(T)$.

If $f(x)=1$, then $f(w)+f(v)=3$. So, we assume that $f(w)=3$ and $f(v)=0$. In such a case, $(g(v),g(w),g(x))=(2,0,2)$ and $g(u)=f(u)$ for the remaining vertices, is a $\gamma_{oidR}(T)$-function. Now, $g'=g\mid_{V(T')}$ is an OIDRD function of $T'$. Therefore, $2\beta(T)+1=2\beta(T')+3\leq \gamma_{oidR}(T')+2\leq \omega(g')+2=\gamma_{oidR}(T)$.

We now suppose that $f(x)=0$. Again, we can assume that $f(w)=3$ and $f(v)=0$. If $f(y)=3$, then $f'=f\mid_{V(T')}$ is an OIDRD function of $T'$ with weight $\gamma_{oidR}(T)-3$. This implies that $2\beta(T)+1<\gamma_{oidR}(T)$. If $f(y)=2$, then $g'(y)=3$ and $g'(u)=f(u)$ for any other vertex $u\in V(T')$, is an OIDRD function of $T'$ with weight $\omega(g')=\gamma_{oidR}(T)-2$. In such a case, we deduce that $2\beta(T)+1\leq \gamma_{oidR}(T)$.

Therefore, in what follows we assume that $f(y)=1$. Note that by our choice of $v$, the vertex $y$ satisfies at least one of the following conditions: $(a)$ $deg(y)=2$; $(b)$ $y$ is adjacent to a (unique) leaf; $(c)$ $y$ has a child which is a support vertex ; or $(d)$ $y$ has a child which is the parent of a support vertex. Hence, we need to consider four possibilities depending on the behavior of $y$.

\textit{Subcase 2.2.1.} Let $y$ be adjacent to a (unique) leaf $y'$. Hence, $f(y')=2$, and so, the assignment $(g'(y'),g'(y))=(0,3)$ and $g'(u)=f(u)$ for any other vertex $u\in V(T')$, is an OIDRD function of $T'$ with weight $\omega(g')=\gamma_{oidR}(T)-3$. Therefore, $2\beta(T)+1=2\beta(T')+3\leq \gamma_{oidR}(T')+2\leq \omega(g')+2<\gamma_{oidR}(T)$.

\textit{Subcase 2.2.2.} Let $y$ have a child $y'$ which is a support vertex, and let $y''$ be the unique leaf adjacent to $y'$. Hence, we can assume that $f(y')=3$ and $f(y'')=0$. We then conclude that $(g'(y''),g'(y'),g'(y))=(2,0,3)$ and $g'(u)=f(u)$ for the remaining vertices $u\in V(T')$, is an OIDRD function of $T'$ with weight $\omega(g')=\gamma_{oidR}(T)-2$. We consequently deduce that $2\beta(T)+1=2\beta(T')+3\leq \gamma_{oidR}(T')+2\leq \omega(g')+2=\gamma_{oidR}(T)$.

\textit{Subcase 2.2.3.} Let $y$ have a child $y'$ which is adjacent to a support vertex $y''$, and let $y'''$ be the unique leaf adjacent to $y''$. Then, $3\leq f(y')+f(y'')+f(y''')\leq4$. Suppose first that $f(y')+f(y'')+f(y''')=4$. We may assume that $f(y')=f(y''')=2$ and $f(y'')=0$. Then, the assignment $(g'(y'''),g'(y''),g'(y'),g'(y))=(0,3,0,3)$ and $g'(u)=f(u)$ for any other vertex $u\in V(T')$, is an OIDRD function of $T'$ with weight $\omega(g')=\gamma_{oidR}(T)-2$, and so we obtain $2\beta(T)+1\leq \gamma_{oidR}(T)$ similarly to Subcase 2.2.2.

If $f(y')+f(y'')+f(y''')=3$, then we have $f(y''')=f(y')=0$ and $f(y'')=3$, necessarily. In such a situation, we consider the subtree $T''=T-T_{w}-T_{y''}$. It is easy to see that $\beta(T)=\beta(T'')+2$. On the other hand, the assignment $g'(y)=3$ and $g'(u)=f(u)$ for the other vertices $u\in V(T'')$ is an OIDRD function of $T''$ with weight $\omega(g')=\gamma_{oidR}(T)-4$. Therefore, $2\beta(T)+1=2\beta(T'')+5\leq \gamma_{oidR}(T'')+4\leq \omega(g')+4=\gamma_{oidR}(T)$.

\textit{Subcase 2.2.4.} We now consider the situation in which $deg(y)=2$. Since $diam(T)\geq4$, the vertex $y$ has a parent $z$. Moreover, we must have $f(z)\geq2$. We observe that the assignment $g'(x)=2$, $g'(y)=0$ and $g'(u)=f(u)$ for any remaining vertex $u\in V(T')$, is an OIDRD function of $T'$ with weight $\omega(g')=\gamma_{oidR}(T)-2$, and we deduce that $2\beta(T)+1\leq \gamma_{oidR}(T)$ by a similar fashion.

This completes the proof of the lower bound. To see the tightness of it, we consider the path graphs of even order, since $\gamma_{oidR}(P_{2t})=2t+1=2\beta(P_{2t})+1$ (by using Observation \ref{ob1} (i)).
\end{proof}


\section{Corona graphs}

Let $G$ and $H$ be graphs where $V(G) = \{v_1, \ldots ,v_{n}\}$. We recall that the corona $G\odot H$ of graphs $G$ and $H$ is obtained from the disjoint union of $G$ and $n$ disjoint copies of $H$, say $H_1,\ldots, H_{n}$, such that for all $i\in \{1,\dots,n\}$, the vertex $v_i\in V(G)$ is adjacent to every vertex of $H_i$. We next present an exact formula for $\gamma_{oidR}(G\odot H)$ when $\Delta(H)\leq|V(H)|-2$.

\begin{theorem}
Let $G$ be a graph of order $n$, and let $H$ be a graph of maximum degree at most its order minus two. Then $\gamma_{oidR}(G\odot H)$ equals
$$\min\{|V_0|(n(H)+\gamma(H))+|V_1|(\gamma_{oidR}(H)+1)+|V_2|(\gamma_{oiR}(H)+2)+|V_3|(\beta(H)+3)\},$$
taken over all possible functions $f_G=(V_0,V_1,V_2,V_3)$ over $V(G)$ for which the vertices labeled with $0$ form an independent set.
\end{theorem}

\begin{proof}
Consider a function $f_G=(V_0,V_1,V_2,V_3)$ over $V(G)$ such that the vertices labeled with $0$ form an independent set. We next describe a function $f:V(G\odot H)\rightarrow\{0,1,2,3\}$ defined in the following way. Let $v_i\in V(G)=\{v_{1},\cdots,v_{n}\}$.
\begin{itemize}
  \item If $f_G(v_i)=0$, then we take a $\gamma(H)$-set $D$, and for every vertex $w\in V(H_i)$ we make $f(w)=2$ if $w\in D$, and $f(w)=1$ otherwise.
  \item If $f_G(v_i)=1$, then we choose a $\gamma_{oidR}(H)$-function $f_H$ and for every vertex $w\in V(H_i)$ we make $f(w)=f_H(w)$.
  \item If $f_G(v_i)=2$, then we choose a $\gamma_{oiR}(H)$-function $g_H$ and for every vertex $w\in V(H_i)$ we make $f(w)=g_H(w)$.
  \item If $f_G(v_i)=3$, then we take an $\alpha(H)$-set $S$, and for every vertex $w\in V(H_i)$ we make $f(w)=0$ if $w\in S$, and $f(w)=1$ otherwise.
  \item For every vertex $v_i\in V(G)$, we make $f(v_i)=f_G(v_i)$.
\end{itemize}
We shall now prove that such function $f$ is an OIDRD function of $G\odot H$. We consider several situations for a given $i\in \{1,\dots,n\}$.

\begin{itemize}
  \item $f_G(v_i)=0$. Since $H$ has maximum degree at most its order minus two, the $\gamma(H)$-set $D$ has at least two vertices. Thus, $v_i$ has at least two neighbors labeled with $2$. Moreover, every vertex $w\in V(H_i)$ such that $f(w)=1$ has a neighbor labeled with $2$ since $D$ is a dominating set of $H$.
  \item $f_G(v_i)=1$. Since $f_H$ is a $\gamma_{oidR}(H)$-function, every vertex of $V(H_i)$ satisfies the condition for $f$ to be an OIDRD function in $G\odot H$. Among other things, this also means that there is at least one vertex in $V(H)$ labeled with $2$ or $3$ under $f_H$. So, the vertex $v_i$ is adjacent to at least one vertex with label $2$ or $3$.
  \item $f_G(v_i)=2$. Note that any vertex of $V(H_i)$, labeled with $0$ under $g_H$, is adjacent to a vertex labeled with $2$ in $V(H_i)$. Also, since every vertex of $V(H_i)$ is adjacent to $v_i\in V(G)$, and $f(v_i)=2$, it follows that every vertex labeled with $0$ is adjacent to at least two vertices labeled with $2$, as well as, every vertex labeled with $1$ is adjacent to at least one vertex labeled with $2$.
  \item $f_G(v_i)=3$. Since every vertex of $V(H_i)$ is adjacent to $v_i$, it clearly follows that every vertex of $V(H_i)$ satisfies the condition for $f$ to be an OIDRD function of $G\odot H$.
\end{itemize}
As a consequence of all the situations described above, we deduce that $f$ is an OIDRD function of $G\odot H$. Since this has been made for an arbitrary function $f_G=(V_0,V_1,V_2,V_3)$ over $V(G)$ such that the vertices labeled with $0$ form an independent set, it is in particular satisfied for that function which gives the minimum weight. Furthermore, $\alpha(H)+\beta(H)=n(H)$. Therefore, $\gamma_{oidR}(G\odot H)\le \min\{|V_0|(n(H)+\gamma(H))+|V_1|(\gamma_{oidR}(H)+1)+|V_2|(\gamma_{oiR}(H)+2)+|V_3|(\beta(H)+3)\}$.

On the other hand, consider a $\gamma_{oidR}(G\odot H)$-function $g=(V'_0,V'_1,V'_2,V'_3)$ and let $v_i\in V(G)$. We analyze now some cases.\\

\noindent
\emph{Case 1:} $g(v_i)=0$. Hence, for every vertex $u\in V(H_i)$ it follows, $g(u)\ge 1$. Moreover, there must be at least one vertex $w\in V(H_i)$, such that $g(w)\ge 2$, since every vertex labeled $1$ under $g$ must be adjacent to a vertex labeled with $2$ or $3$ under $g$. Thus, it follows that $(V'_2\cup V'_3)\cap V(H_i)$ is a dominating set of $H_i$, and so, $g(V(H_i)\cup\{v_i\})\ge 2|(V'_2\cup V'_3)\cap V(H_i)|+|V'_1\cap V(H_i)|\ge 2\gamma(H)+n(H)-\gamma(H)=n(H)+\gamma(H)$.\\

\noindent
\emph{Case 2:} $g(v_i)=1$. In such a situation, it can be readily seen that the restriction of $g$ over $H_i$ must be an OIDRD function of $H_i$. Thus, $g(V(H_i)\cup\{v_i\})\ge \gamma_{oidR}(H)+1$.\\

\noindent
\emph{Case 3:} $g(v_i)=2$. Since every vertex of $V(H_i)$ is adjacent to $v_i$, the condition for a vertex $u\in V(H_i)$ (labeled with $0$) requiring to have two adjacent vertices labeled with $2$ (if it is the case), implies that at least one of such neighbors must be in $V(H_i)$. Also, note that if there exists a vertex $w\in V(H_i)$ such that $g(w)=3$, then we can redefine $g(w)$ as $g(w)=2$ (maintaining all the remaining labels the same), and we obtain an OIDRD function of $G\odot H$ with smaller weight, which is not possible. Thus, every vertex of $V(H_i)$ has label at most $2$. Consequently, the restriction of $g$ over $H_i$ must be an OIRD function of $H_i$. Therefore, $g(V(H_i)\cup\{v_i\})\ge \gamma_{oiR}(H)+2$.\\

\noindent
\emph{Case 4:} $g(v_i)=3$. Now, we can easily observe that for every vertex $w\in V(H_i)$, it must happen $g(w)\le 1$. Since $V(H_i)\cap V'_0$ is an independent set, we obtain that $g(V(H_i)\cup\{v_i\})\ge n(H)-\alpha(H)+3=\beta(H)+3$.\\

Since $V'_0$ is an independent set, it is clear that the function $f'_G=(V''_0,V''_1,V''_2,V''_3)=(V'_0\cap V(G),V'_1\cap V(G),V'_2\cap V(G),V'_3\cap V(G))$ satisfies that $V''_0=V'_0\cap V(G)$ is independent. As a consequence of all the cases above, by making the sum $\sum_{i=1}^n g(V(H_i)\cup\{v_i\})$, we deduce that 
\begin{align*}
  \gamma_{oidR}(G\odot H) & \ge |V''_0|(n(H)+\gamma(H))+|V''_1|(\gamma_{oidR}(H)+1)+|V''_2|(\gamma_{oiR}(H)+2)+|V''_3|(\beta(H)+3)\\
   & \ge \min\{|V_0|(n(H)+\gamma(H))+|V_1|(\gamma_{oidR}(H)+1)\\
   &\hspace*{0.0cm}+|V_2|(\gamma_{oiR}(H)+2)+|V_3|(\beta(H)+3)\},
\end{align*}
taken over all possible functions $f_G=(V_0,V_1,V_2,V_3)$ over $V(G)$ for which the vertices labeled with $0$ form an independent set. This completes the proof.
\end{proof}

\end{document}